\documentclass[12pt]{article}
\usepackage{amsmath,amssymb,amsthm,latexsym,euscript,amscd, authblk}
\usepackage[english]{babel}
\usepackage[english]{babel}
\usepackage[all]{xy}
\usepackage[dvips]{graphics}

\newtheorem{Theorem}{Theorem}
\newtheorem{Proposition}[Theorem]{Proposition}
\newtheorem{Corollary}[Theorem]{Corollary}
\newtheorem{Lemma}[Theorem]{Lemma}
\newtheorem{Example}[Theorem]{Example}

\title{Homotopes of finite-dimensional algebras.}

\author[1, 2]{I.Yu. Zhdanovskiy\thanks {ijdanov@mail.ru}}

\affil[1]{Moscow Institute of Physics and Technology, Laboratory AGHA}
\affil[2]{National Research University High School of Economics, Laboratory of Algebraic Geometry}

\begin{document}

\maketitle

\begin{abstract}
In this article we study homotopes of finite-dimensional algebras (not necessarily, associative). In the case of associative algebras we study homotopes by methods of Category theory and give description of so-called well-tempered elements of finite-dimensional associative algebra in algebraic terms.
\end{abstract}

\section{Introduction}

%Non-associative algebras are very wide class of algebras. This class is far from complete understanding even in finite-dimensional case.
%Of course,

Non-associative algebras is a branch of algebra which is far from complete understanding even in finite-dimensional case. There are many examples of applications of non-associative algebras in many branches of math and physics: Lie algebras, Jordan algebras etc.
Important concepts for studying of non-associative algebras are isotopy and homotopy of algebras. Isotopy and homotopy were introduced by Albert (see \cite{Alb}) and intensively studied by many investigators (see \cite{Bruck}, \cite{McC}, \cite{Jac} etc).

Isotopy permits us to deform of multiplication structure of algebras. Namely, consider two algebras $(A,m_A)$ and $(B,m_B)$ with multiplication laws $m_A$ and $m_B$ respectively. Homotopy (isotopy) of algebras $A$ and $B$  is a three linear (bijective linear) maps $f_1,f_2,f_3: A \to B$ such that $f_3(m_A(a_1,a_2)) = m_B(f_1(a_1),f_2(a_2))$. Of course, the concept of homotopy(isotopy) is a generalization of the concept of homomorphism (isomorphism) of algebras.

One of the interesting partial cases of homotopy is a concept $a$-homotope of algebra. Consider algebra $(A,m)$ and fix element $a \in A$. Define new multiplication laws $L_a(m)$ and $R_a(m)$ by formulas: $L_a(m)(x,y) = m(x,m(a,y))$ and $R_a(m)(x, y) = m(m(x,a),y)$ respectively. Algebras $(A,L_a(m))$ and $(A,R_a(m))$ are called left and right homotopes of $A$ with respect to element $a$. Also, there are many generalizations of the notion $a$-homotope in partial cases: Jordan algebras, alternative algebras etc.

In our work we will consider $a$-homotopes of associative algebras.
Of course, if algebra $(A,m)$ is associative then $L_a(m) = R_a(m)$ for any element $a \in A$.
Denote by $A_a$ the homotope of $A$ with respect to element $a$. Adding unit element to $A_a$ externally, we get {\it augmented homotope} $\widehat{A}_a$.

Augmented homotopes play important role in Quantum Information Theory. Firstly, recall that one of important notion of Quantum Information Theory is mutually unbiased bases. These bases were introduced by Schwinger (cf. \cite{Schw}) and used in construction of quantum protocol BB84 (see \cite{BB84}). Two orthonormal bases $\{e_i\}^n_{i=1}$ and $\{f_j\}^n_{j=1}$ in Hermitian vector space are mutually unbiased if $|(e_i, f_j)|^2 = \frac1n$ for any $i,j$. One of the hardest problem of Quantum Information Theory is a classification of MUBs. This problem is completely solved only if $n \le 5$.

Explain the role of augmented homotopes in the problem of classification of MUBs.
For this purpose, recall the notion of reduced Temperley-Lieb algebra (see \cite{TL}, \cite{BZ}).
Let $\Gamma$ be a simply-laced graph. Reduced Temperley-Lieb algebra $B_r(\Gamma)$ be an algebra over ${\mathbb C}[r,r^{-1}]$ with generators $x_v$ labeled by vertices of $\Gamma$. These generators subject to the following relations: $x^2_v = x_v$, $x_v x_w = x_w x_v = 0$ if $(v,w)$ is not an edge in $\Gamma$ and $x_v x_w x_v = r x_v, x_w x_v x_w = r x_w$ if $(v,w)$ is an edge in $\Gamma$.
Consider complete bipartite graph $K_{n,n}$. In this case classifications of $n$-dimensional representations of $B_r(K_{n,n})$ is a "complexification" of the problem of classification of MUB's in ${\mathbb C}^n$.
In \cite{BZ} algebra $B_r(\Gamma)$ is intensively studied and, in particular, it was shown that algebra $B_r(\Gamma)$ is a homotope of path algebra of graph $\Gamma$ with respect to laplacian $\Delta$ of $\Gamma$. Using general theory developed in \cite{BZ}, it was shown the existence of four-dimensional family of MUBs in dimension $6$ (cf. \cite{BZ6}) and description of one-dimensional family in algebraic terms in dimension $7$ (see \cite{KZ1}, \cite{KZ2}).

This work was motivated by the following concept of homotopes of associative algebras. Recall the notion {\it well-tempered} elements \cite{BZ}. Let $A$ be a unital associative algebra $A$. Fix element $x \in A$. Consider augmented homotope $\widehat{A}_x$. One can construct two homomorphisms of unital algebras: $\psi_i: \widehat{A}_x \to A, i = 1,2$ defined by rules: $\psi_1: a \mapsto ax$ and $\psi_2: a \mapsto xa$. We will say that $x$ is a {\it well-tempered} if $x$ subject to the following conditions:
\begin{enumerate}
\item{$A$ is a projective left and right $\widehat{A}_x$-module, where structure of left (resp. right) $\widehat{A}_x$ - module is obtained from $\psi_1$ (resp. $\psi_2$)}
\item{$AxA = A$.}
\end{enumerate}
In \cite{BZ} it was proven that if $x$ is well-tempered then abelian categories ${\rm k - Mod}$ and ${\rm A - Mod}$ are full subcategories of ${\rm \widehat{A}_x - Mod}$. Moreover, ${\rm k - Mod}$, ${\rm \widehat{A}_x - Mod}$ and ${\rm A - Mod}$ are in recollement situation in the sense of \cite{McV}, \cite{Psa}. This situation permits to study ${\rm \widehat{A}_x - mod}$ in terms of ${\rm A - mod}$. In particular, one can get the estimation of global homological dimension of ${\rm \widehat{A}_x - mod}$ in terms of global dimension of ${\rm A - mod}$.

Main result of this article is the following statement:
\begin{Theorem}
Consider finite-dimensional associative algebra $A$ and element $x \in A$ such that $A x A = A$ then $x$ is well-tempered element of $A$.
\end{Theorem}
Note that if $A$ is a finite-dimensional algebra and $A$ is left and right projective $\widehat{A}_x$ - module, then $AxA = A$.
Thus, we get the complete description of well-tempered elements in terms of two-sided ideals of the algebra $A$.
Also, note the following property of non-well-tempered elements: if $x$ is not well-tempered element then global homological dimension of $\widehat{A}_x$ if $x$ is infinite.

Our article is organized as follows. There are two sections of the article.
Firstly, we recall the concepts of homotopy and isotopy of algebras. We remind the classification of Bruck and using Popov's result \cite{Pop} we get that generic algebra is left-simple and right-simple. Further, we recall the notion of $a$-homotope and formulate some results on it. Second chapter is devoted to study well-tempered elements in finite-dimensional case. In particular, proof of the main theorem is in this section. Last parts of this chapter are devoted to commutative algebras. In this case the notion of well-tempered elements is trivial. Thus, we consider homotopes of infinite-dimensional noetherian algebras with respect to non-well-tempered element. In this case ${\rm A - Mod}$ is not a full subcategory of ${\rm \widehat{A}_x - Mod}$, but there is a common full subcategory of ${\rm A - mod}$ and ${\rm \widehat{A}_x - mod}$.

{\bf Acknowledgements.} I am grateful to Alexei Bondal and Ilya Karzhemanov for very fruitful discussions and support. Author was partially supported by the HSE University Basic Research Program and the Russian Academic Excellence Project '5-100 and partially supported by grant RFBR - 18-01-00908

%\section{Preliminary remarks.}

\section{Homotopy of algebras: previous remarks.}
\subsection{Isotopy of algebras. Generic algebras.}
\label{homotop}

Firstly, recall the notion of homotopy and isotopy of algebras.
%Algebras are not presumed to be associative.

Fix $d$-dimensional vector space $V$ over algebraically closed field $k$ of characteristic zero.
It is easy that tensor $m \in {\cal M} =  V^* \otimes V^* \otimes V = {\rm Hom}_k(V\otimes V, V)$ defines multiplication law.
Denote by $(V,m)$ the algebra with fixed multiplication law $m$.
%There is a well-defined action of ${\bf G} = {\rm GL}(V)^{\times 3}$ on ${\cal M}$.
%Classification of the orbits of the action ${\bf G}$ on ${\cal M}$ is very hard problem which is completely solved only in the case $d \le 3$ \cite{BP}, \cite{Ng}.

%From other hand,
Albert \cite{Alb} introduced the notion of isotopy of algebras as follows. Algebras $(V,m_1)$ and $(V,m_2)$ are {\it isotopic} iff there are bijective linear maps $f_i \in {\rm Aut}(V), i = 1,2,3$ such that $f_1(m_1(v',v'')) = m_2(f_2(v'),f_3(v'')$ for any $v',v'' \in V$.
Also, there is a notion of homotopy of two algebras. Namely, algebras $(V,m_1)$ and $(V,m_2)$ are {\it homotopic} iff there is a set of three linear maps: $f_i \in {\rm End}_k(V), i = 1,2,3$ such that $f_1(m_1(v',v'')) = m_2(f_2(v'),f_3(v''))$ for any $v',v'' \in V$. It is easy that the notions of homotopy and isotopy are generalizations of homomorphism and isomorphism of algebras. Actually, if $f_1 = f_2 = f_3 = f$ then $f: V \to V$ is a homomorphism (if $f$ is bijective then $f$ is isomorphism) of algebras.

Consider group ${\bf G} = GL(V)^{\times 3}$ with natural action on ${\cal M}$.
It is easy that isotopic classes are in bijection with ${\cal M}/{\bf G}$. If we consider diagonal group $G \subset {\bf G}$ and its action on ${\cal M}$, we get that isomorphic classes of algebras on $V$ are in bijection with points ${\cal M}/G$.
%There is a well-defined action of ${\bf G} = {\rm GL}(V)^{\times 3}$ on ${\cal M}$.
Note that classification of the orbits of the action ${\bf G}$ on ${\cal M}$ is very hard problem which is completely solved only in the case $d \le 3$ (cf. \cite{BP}, \cite{Ng}).

%If $f_i$ are bijective then $\bf f$ is called isotopy.
%Bruck introduced the following classification of algebras:

Recall the following well-known results about isotopy of algebras.
Consider algebra $(V,m)$. If there are elements $a,b \in V$ such that $l_a = m(a,-)$ and $r_b = m(-,b)$ are invertible operators, then algebra $(V,m)$ is isotopic to some unital algebra $(V,m')$. Actually, define $m'$ by formula: $m'(x,y) = m(r^{-1}_b(x), l^{-1}_a(y))$. One can check that $m(a,b)$ is a unit of algebra $(V,m')$.  Note that converse statement is true. These statements are called by {\it Kaplanski's trick}.
It is clear that algebras may be divided into four classes under isotopy (see \cite{Bruck}):
\begin{enumerate}
    \item{Algebras with at least one left invertible and one right invertible element}
    \item{Algebras with at least one left invertible but with no right invertible element}
    \item{Algebras with no left invertible but at least one right invertible element}
    \item{Algebras with no left invertible and with no right invertible element.}
\end{enumerate}
%Denote by $Z_i, i = 1,2,3,4$ the sets parameterizing algebras of $i$th class, respectively. %It is easy that $Z_i$ are ${\bf G}$ - sets.

%Consider ${\cal M} = {\rm Hom}_k(V \otimes V,V)$.
We have the following maps: $l, r:{\cal M} \times V \to {\rm End}_k(V)$ given by formulas: $(m,v) \mapsto r_v = m(-,v)$ and $(m,v) \mapsto l_v = m(v,-)$. It is easy that this morphism is surjective.
Denote by ${\cal M}_1$ the set of $m$ such that $(V,m)$ is an algebra of first type.
Standard arguments of algebraic geometry give us the following proposition:
\begin{Proposition}
\label{ginv}
\begin{itemize}
    \item{${\cal M}_1 \subset {\cal M}$ is Zarisski-open dense subset}
    \item{For fixed $m \in {\cal M}_1$ and generic $v \in V$ operators $l_v = m(v,-)$ and $r_v = m(-,v)$ are invertible.}
    \item{For fixed nonzero $v \in V$ there is Zarisski-open dense subset ${\cal M}(v) \subset {\cal M}$ such that $l_v = m(v,-)$ and $r_v = m(-,v)$ are invertible.
    }
\end{itemize}
\end{Proposition}

Associative algebras play important role in isotopy classes of algebras:
\begin{Proposition}(cf. \cite{Alb})
\label{assiso}
If unital algebra $(V,m)$ is isotopic to unital associative algebra $(V,m')$. Then $(V,m)$ and $(V,m')$ are isomorphic.
\end{Proposition}
Note that we can reformulate this statement in the terms of ${\cal M}_1$ and variety of unital associative algebras ${\cal A}$:
natural map: ${\cal A}/G \to {\cal M}_1/{\bf G}$ is injective.

Using proposition \ref{ginv}, we get the following
\begin{Corollary}
Generic algebra $(V,m)$ is isotopic to unital algebra.
\end{Corollary}

Further, consider generic algebras. For this purpose, fix $m \in {\cal M}$.
Let ${\cal L}(m)$, ${\cal R}(m)$ and ${\cal U}(m)$ be an associative algebras of generated by space $l_v, v \in V$, $r_v, v \in V$ and both $l_v, v \in V$ and $r_v, v \in V$ respectively.
Algebra $(V,m)$ is said to be {\it right-simple} iff it contains no proper right ideals. {\it Left-simplicity} and {\it simplicity} are defined analogously.

It is easy that algebra $(V,m)$ is right-simple, (left-simple or simple) iff $V$ is a simple ${\cal R}(m)$ (${\cal L}(m)$ or ${\cal U}(m)$) - module.
Popov proved (see \cite{Pop}) that generic algebra is simple. Repeating his proof with small changes, we get the following proposition:
\begin{Proposition}
\label{gensimp}
Generic algebra $(V,m)$ is left-simple and right-simple.
\end{Proposition}
\begin{proof}
Put ${\cal M}(r)$ the set of $m$ such that $(V,m)$ has r-dimensional left ideal.
Pickup basis $v_1,...,v_d$ of $V$. Let $v^1,...,v^d$ is a dual basis of $V^*$.
Let $V_r$ be linear span of $v_1,...,v_r$. Let $M(r)$ be the set of algebras $(V,m)$  such that $V_r$ is a left ideal of $(V,m)$. It is easy that ${\cal M}(r) = G \cdot M(r)$, where $G$ is a diagonal subgroup of ${\bf G}$.
%Let us compute the dimension of $M(r)$.
Let $m = \sum c^l_{ij} e^i \otimes e^j \otimes e_l$. Since $V_r$ is a left ideal then $c^s_{ij} = 0$ for $s > r$ and $j \le r$. Thus, ${\rm dim}_kM(r) = d^3 - rd^2+r^2d$. Further, stabilizer of $M(r)$ is a parabolic subgroup $P_r$ of $G$ of dimension $d^2-rd+r^2$. Therefore, ${\rm dim}_k{\cal M}(r) \le {\rm dim}_kG - {\rm dim}_kP_r + {\rm dim}_kM(r) = d^2 - (d^2-dr+r^2)+d^3-rd^2+r^2d = d^3 - r(d-1)(d-r) < d^3$ if $d > 1$ and $r \ge 1$ and $r < d$.
One can prove that generic algebra is right-simple analogously.
\end{proof}

Algebra $(V,m)$ is {\it isotopically left(right)-simple} if any isotope of $(V,m)$ is left(right)-simple.
Recall the following theorem of Bruck (see \cite{Bruck}):
\begin{Proposition}
If right-simple(left-simple) algebra $(V,m)$ has right unit (left unit) then $(V,m)$ is isotopically right-simple (left-simple).
\end{Proposition}
\begin{proof}
Let $(V,m)$ be a left-simple algebra.
It is easy that one can consider only principal isotopes. Consider principal isotope $(V,m')$ defined by rule: $m'(x,y) = m(g^{-1}_2(v),g^{-1}_3(v))$ for some $g_2,g_3 \in GL(V)$. Consider algebra ${\cal L}(m')$. It is easy that $l'_v = m'(v,-) = l_{g^{-1}_2(v)} \circ g^{-1}_3 \in {\rm End}_{\mathbb F}(V)$ for $v \in V$. Since $(V,m)$ is unital algebra, we get that $g_3 \in {\cal L}(m')$. Thus, ${\cal L}(m) \subset {\cal L}(m')$ and hence, $V$ is a simple ${\cal L}(m')$ - module. q.e.d.
\end{proof}
Using this proposition, proposition \ref{gensimp} and proposition \ref{ginv}, we get that
\begin{Corollary}
Generic algebra $(V,m), m \in {\cal M}$ is isotopically left and right-simple.
\end{Corollary}

%Сlass of unital associative algebras plays important role in isotopy of algebras.
%Actually, Albert proved the following property:
%\begin{Proposition}
%If unital algebra $(V,m)$ is isotopic to unital %associative algebra $(V,m')$. Then $(V,m)$ is %isomorphic to $(V,m')$. Unital associative %algebra may be isotopic to non-unital %non-associative algebra.
%\end{Proposition}
This corollary demonstrates the difference between associative algebras and non-associative algebras. There are many non-isotopic simple non-associative algebras but simple unital associative algebra is one up to isotopy.
%Denote by ${\cal M}^{ass}(V)$ the variety parameterizing of unital associative multiplication law on $V$.
Also, it is well-known that affine scheme parameterized unital associative algebras is reducible. Asymptotically, dimension of any components is less or equal than $\frac{4}{27}d^3 + o(d^3)(d \to \infty)$ (see \cite{Ner}). It was shown for any $d$ there is the component consisting of metabelian algebras, constructed by Vergne (see \cite{Verg}). Asymptotic of dimension of this component is the same. Also, recall that simple unital associative algebra is rigid, and hence, component corresponding to simple algebra has dimension $d^2-1$. It means that generic unital associative algebra is not simple.

\subsection{$a$-homotopes.}

%In this subsection we study property of variety parameterizing unital algebras. Also, we study homotopes of special type.
In this subsection we introduce the partial case of homotopy - $a$-homotope.

Fix algebra $(V,m)$ and $a \in V$. Define maps $L(a), R(a): {\cal M} \to {\cal M}$ by the following formulas:
\begin{equation}
L(a)m(v',v'') := m(v',l_{a}(v'')), R(a)m(v',v'') := m(r_{a}(v'),v'')
\end{equation}
Algebras $(V,L(a)m)$ ($(V,R(a)m))$ are called by {\it left $a$-homotope}({\it right $a$-homotope}) of algebra $(V,m)$ with respect to element $a$. For simplicity, we will call left or right $a$-homotope by homotope if it does not lead to confusion. Sometimes, $a$-homotope is called {\it mutation}.
These notions were established in fifties and studied by many investigators.  Of course, if element $a$ is left-invertible (or right-invertible) then algebra $(V,L(a)m)$ (or $(V,R(a)m)$) and $(V,m)$ are isotopic.

Note that there are several generalizations of the notion of homotope. For example, $(u,v)$ - homotope (see \cite{McC}) is an algebra with multiplication law defined by formula: $m'(v',v''):= m(r_u(v'),l_v(v''))$, $(a,b)$ - mutation $S(a,b)$ defined by rule: $S(a,b)m(v',v'') = L(a)m(v',v'') - R(b)m(v',v'')$ etc.

%Consider endomorphism of $V$ of special type. For this purpose, fix non-zero $v \in V$.

%Using operators $l_v$ and $r_v$ one can define the following morphisms $R(v), L(v): {\cal M} \to {\cal M}$ defined by rules: $R(v): m \mapsto m(r_v(-),-), L(v): m \mapsto m(-,l_v(-))$.
%Using proposition \ref{ginv}, we obtain that $L(v) (m)$ and $R(v) (m)$ are isogenous to $m$ for generic $m \in {\cal M}$. %Thus, $R(v)$ and $L(v)$ are morphisms of ${\cal M}$.
%Note the following properties of $L(v)$. Fix $g \in GL(V)$. Put $m' = (1,1,g) \circ m$, i.e. $m'(x,y) = m(x,g^{-1}(y))$.
%$$
%R(g(v)) \circ (1,1,g) m (x,y) = R_{g(v)} m'(x,y) = m'(m'(x,g(v)),y) %= m(m(x,v),g^{-1}(y)).
%$$
%One can show that $(1,1,g) \circ R(v) m(x,y) = m(m(x,v),g^{-1}(y))$.
%It means that $R(g(v)) \circ (1,1,g) = (1,1,g) \circ R(v)$ for any $g \in GL(V)$ and $v \in V\setminus{0}$. Analogously, $L(g(v)) \circ (1,g,1) = (1,g,1) \circ L(v)$. Therefore, geometric properties of $L(v), v \in V \setminus{0}$ are the same.

Note the following property of morphisms $R(v)$ and $L(v)$ for nonzero $v \in V$:
\begin{Proposition}
For $v \in V\setminus{0}$ morphisms $R(v)$ and $L(v)$ are dominant. Moreover, ${\rm deg}L(v) = {\rm deg} R(v) = 2^d$.
\end{Proposition}
\begin{proof}
It is sufficient to prove that ${\rm deg}R(v) = 2^d$. Fix nonzero $v \in V$.
Let us prove that $|R(v)^{-1}(m')| = 2^d$ for generic $m' \in {\cal M}$.

Assume that $m' = R(v)(m)$. Fix basis $v_1 = v,...,v_d$ of $V$. %Assume that $v = \sum^d_{i=1} x_i v_i \in V$.
Denote by $R_i, R'_i$ the operators $m(-,v_i), m'(-,v_i) \in {\rm End_k}(V)$.
One can deduce that
\begin{equation}
\label{rrr}
R'_i = R_i R_1 , i = 1,..,d.
\end{equation}

Further, solve equations (\ref{rrr}) for generic $R'_i$.
Denote by $S = \{ s \in GL(V)| \lambda_i(s) \ne \lambda_j(s), i,j = 1,...,d\}$ where $\lambda_i(s), i = 1,...,d$ are eigenvalues of $s$. It is easy that $S$ is a dense open subvariety of $GL(V)$ (and hence, ${\rm End}_k(V)$).
%Assume that $x_1 \ne 0$ and $R'_1 \in S$. Let us rewrite the system in the following way:
%$$
%(\sum^d_{i=1} x_i R_i)(\sum^d_{i=1}x_i R_i) = R'_1,
%$$
%$$
%R_i(\sum^d_{i=1} x_i R_i) = R'_i, i = 2,...,d
%$$
One can show that if $R'_1 \in S$ then there are $2^d$ solutions of first equation. Also, one can find $R_i, i = 2,...,d$ uniquely from another equations. Analogous arguments prove the rest.
%Standard arguments give us that there are $2^d$ solutions of first equation. Also, one can find $R_i, i = 2,...,d$ uniquely from another equations. Analogous arguments prove the rest.
\end{proof}

Note the following categorial description of homotopes.
Let ${\rm Alg}$ be a category of algebras (not necessary associative).%, i.e. algebra is a pair $(V,m)$, where $V$ is a vector space, $m: V \otimes V \to V$ is a bilinear map.
Morphisms of ${\rm Alg}$ are homomorphisms of algebras.
Consider category ${\cal C}$ defined as follows. Objects of ${\cal C}$ are pairs $(A,a)$, where $A \in {\rm Alg}$, $a \in A$ is an element of $A$. Morphism $\phi: (A,a) \mapsto (A',a')$ is a morphism of algebras such that $\phi(a)= a'$. Consider map $L: {\cal C} \to {\rm Alg}$ defined by formula:
$L: (A,a) \mapsto A'$, where $A'$ is left homotope of $A$ with respect to $a$. Analogously, one can define map $R: {\cal C} \to {\rm Alg}$, where $R(A,a)$ is right homotope of $A$ with respect to $a$.

\begin{Proposition}
\label{commhom}
\begin{itemize}
\item{Maps $L,R: {\cal C} \to {\rm Alg}$ are well-defined functors.}
\item{If $I$ is a two-sided ideal of algebra $A$, then $I$ is a two-sided ideal of algebras $L(A,a)$ and $R(A,a)$ for any $a \in A$.}
\item{Consider algebra $A$, its quotient $A/I$ by two-sided ideal $I$ and natural morphism: $\phi: A \to A/I$. Then we have the isomorphism of algebras: $L(A,a)/I \cong L(A/I, \phi(a))$.}
\end{itemize}
\end{Proposition}
\begin{proof}
Straightforward.
\end{proof}
%It can be shown in usual way that  $L, R: {\cal C} \to {\rm Alg}$ are functors.
%Note that if $I$ is a two-sided ideal of algebra $A$ then $I$ is a two-sided ideal of $L(A,a)$ and $R(A,a)$ for any $a \in A$.
%Consider algebra $A$, its quotient $A/I$ by two-sided ideal $I$ and natural morphism: $\phi: A \to A/I$. One can show that $L(A,a)/I \cong L(A/I, \phi(a))$.

%Denote by ${\rm Assoc}$ and ${\rm Com}$ the categories of associative and commutative associative $k$-algebras, respectively. Consider associative algebra $A$.
%It is easy that $L(A,a) = R(A,a)$ for any element $a \in A$. For simplicity, denote by $A_a$ the algebra $L(A,a) = R(A,a)$ for associative algebra $A$.
%Consider full subcategories ${\cal C'} \subset {\cal C}$ and ${\cal C''} \subset {\cal C'} \subset {\cal C}$ which objects are pairs $(A,a), A \in {\rm Assoc}, a \in A$ and $(A,a), A \in {\rm Com}, a \in A$, respectively.
%Direct checking gives us the following:
%\begin{Corollary}
%Restriction of $L$ (and $R$) to ${\cal C'}$ is a functor $L|_{\cal C'}:{\cal C'} \to {\rm Assoc}$. Analogously, $L|_{\cal C''}: {\cal C''} \to {\rm Com}$. Thus, %$A_a$ is an associative (or commutative associative) algebra for any associative (or commutative associative) algebra $A$ and element $a \in A$
%\end{Corollary}

\section{Homotopes of associative algebras.}
In this section all algebras are presumed associative.

\subsection{Homotopes of associative algebras: previous properties.}
\label{prevprop}

%Consider the case of associative algebras.
%In this case there is a categorial approach

We will write $a \cdot b$ or $ab$ instead of $m(a,b)$ in the case of associative algebras for simplicity.
It is easy that left and right $a$-homotopes are the same for associative algebras. We will call it briefly $a$-homotope (or simply homotope).
Let $A$ be an associative algebra and fix $\Delta \in A$. Consider $\Delta$-homotope $A_{\Delta}$.
Denote by $*_{\Delta}$ the multiplication law of $A_{\Delta}$. It is easy that $A_{\Delta}$ is associative algebra.
%Denote by $A_{\Delta}$ the homotope (left and right) of $A$ with respect to element $\Delta$.
%We will say that vector space $A$ with multiplication law $*_{\Delta}$ is a {\it homotope} of $A$ with respect to $\Delta \in A$. We will denote algebra $(A,*_{\Delta})$ by $A_{\Delta}$.

Using proposition \ref{assiso} and trivial calculations we get the following property of homotopes:
\begin{Proposition}
\label{invele}
Let $A$ be an associative algebra.
\begin{itemize}
    %\item{Algebra $A_{\Delta}$ for any $\Delta \in A$ is associative too.}
    \item{Assume that $c,d$ are invertible elements of algebra $A$. Denote by $\Delta'$ the element $c \cdot \Delta \cdot d$. Then algebras $A_{\Delta'}$ and $A_{\Delta}$ are isomorphic.}
    \item{Assume that $A$ has unit. In this case $A \cong A_{\Delta}$ iff $\Delta$ is invertible element of $A$.
    }
\end{itemize}
\end{Proposition}
\begin{proof}
%First statement is easy.
One can check that morphism $a \mapsto d^{-1}ac^{-1}$ is an isomorphism between $A_{\Delta}$ and $A_{\Delta'}$. The rest is trivial.
\end{proof}
%We will consider unital associative algebras.
Assume that $A$ is unital.
It is clear that if $\Delta$ is not invertible then algebra $A_{\Delta}$ is not unital. Adding the identity element to algebra $A_{\Delta}$ externally, we obtain algebra $\widehat{A}_{\Delta}$. We will call algebra $\widehat{A}_{\Delta}$ by {\it augmented homotope}.

%Consider algebra $\widehat{A}_{\Delta}$ which is isomorphic as vector space to $A \oplus k 1_{\Delta}$ with multiplication $*_{\Delta}$ defined by obvious rule:
%$(a_1+x_1 \cdot 1_{\Delta}) *_{\Delta} (a_2 + x_2 \cdot 1_{\Delta}) = a_1 \Delta a_2 + x_1 a_2 + x_2 a_1 + x_1 x_2 1_{\Delta}, x_i \in k, a_i \in A$.

%For simplicity, denote by $B$ the augmented homotope $\widehat{A}_{\Delta}$ for fixed $\Delta$.
%It is easy that $B^+ = A_{\Delta}$ is a two-sided ideal of $B$. Thus, we have the following exact sequence of $B$ - bimodules:
%\begin{equation}
%\xymatrix{
%0 \ar[rr] && B^+ \ar[rr] && B \ar[rr]^{\epsilon} && k \ar[rr] && 0,
%}
%\end{equation}
%where $\epsilon$ is a natural augmentation. One-dimensional representation of $B$ given by $\epsilon$ will be called by {\it trivial}.

Further, for simplicity, denote by $B$ the augmented homotope $\widehat{A}_{\Delta}$ for fixed $\Delta$.
Consider algebra $B$ as deformation of $A \oplus k \cdot 1$. Since set of invertible elements of unital algebra $A$ is dense Zarisski-open subset of $A$, one can show that augmented homotopes as infinitesimal deformations correspond to zero element of ${\rm HH}^2(A \oplus k \cdot 1, A \oplus k \cdot 1)$. Of course, augmented homotopes as global deformations of $A \oplus k \cdot 1$ may be non-trivial. In the next sections we will study algebras $\widehat{A}_{\Delta}$ for various $\Delta \in A$.

Further, construct morphisms $\psi_i: B = \widehat{A}_{\Delta} \to A, i = 1,2$ as follows. %Consider unital algebra $A$ and fix element $\Delta \in A$. Let $B$ be augmented homotope $\widehat{A}_{\Delta}$.
%We have homotope $B^+ = A_{\Delta}$ and augmented homotope $B: = \widehat{A}_{\Delta}$.
%It is easy that $B^+$ is a two-sided ideal of $B$.
%We have the following exact sequence of $B$ - bimodules:
%\begin{equation}
%\xymatrix{
%0 \ar[rr] && B^+ \ar[rr] && B \ar[rr]^{\epsilon} && k \ar[rr] && 0,
%}
%\end{equation}
%where $\epsilon$ is a natural augmentation.
We have the following identities:
\begin{equation}
\label{assoc}
(b *_{\Delta} a_1)a_2 = b *_{\Delta}(a_1a_2), (a_1 a_2) *_{\Delta} b = a_1(a_2 *_{\Delta} b)
\end{equation}
for any $b \in B$ and $a_1,a_2 \in A$.
Using standard arguments, we get two morphisms of unital algebras: $\psi_i: B \to {\rm End}_A(A) = A, i = 1,2$,
defined by rules:
\begin{equation}
\label{psifs}
\psi_1: a \mapsto a \cdot \Delta, \psi_2: a \mapsto \Delta \cdot a.
\end{equation}
Thus, $\psi_1(B) = k\cdot 1 + A\Delta$ and $\psi_2(B) = k\cdot 1 + \Delta A$, i.e. sum of scalar space and left (or right) principal ideal.

\subsection{Recollement of abelian categories and well-tempered elements.}

In this subsection we recall the notion of recollement of abelian categories. Using \cite{BZ}, we recall the notion of well-tempered element of unital associative algebra and remind that category ${\rm \widehat{A}_{\Delta} - mod}$ is a gluing of categories ${\rm A - mod}$ and ${\rm k - mod}$.

Recollement of categories was introduced first by (see \cite{BBD}) in context of triangulated categories. Recollement situation in abelian categories appeared in the work of McPherson and Vilonen (see \cite{McV}).
Following \cite{Psa}, recall the notion of recollement situation between abelian categories ${\cal A}$, ${\cal B}$ and ${\cal C}$ is a diagram:
\begin{equation}
\xymatrix{
{\cal A} \ar[rrr]_{i}
&&& {\cal B} \ar@/_2pc/[lll]_{p} \ar@/^2pc/[lll]^{q} \ar[rrr]_{e}
&&& {\cal C} \ar@/_2pc/[lll]_{l} \ar@/^2pc/[lll]^{r}
}
\end{equation}
satisfying to conditions:
\begin{itemize}
\item{$(l,e,r)$ is an adjoint triple,}
\item{$(q,i,p)$ is an adjoint triple,}
\item{functors $i$, $l$ and $r$ are fully faithful,}
\item{${\rm Im}i = {\rm Ker}e$}
\end{itemize}

Definition of recollement situation of triangulated categories is the same. Psaroudakis showed that if
\begin{itemize}
\item{abelian categories ${\cal A}, {\cal B}, {\cal C}$ are in recollement situation}
\item{they have enough projective objects}
\item{functors satisfy some natural conditions (see for details \cite{Psa})}
\end{itemize}
then bounded derived categories ${\rm D^b({\cal A})}, {\rm D^b({\cal B})}, {\rm D^b({\cal C})}$ are in recollement situation.

Come back to homotopes. For fixed algebra $A$ denote by ${\rm A - Mod}$ the category of all left $A$-modules.
Consider algebra $A$, element $\Delta \in A$ and homotope $B = \widehat{A}_{\Delta}$.
In the work \cite{BZ} it was shown that if element $\Delta$ satisfy to some natural conditions then abelian categories ${\rm k - Mod}$, ${\rm B - Mod}$ and ${\rm A - Mod}$ are in recollement situation.
Remind the proof of this fact. For this purpose, recall the following exact sequence of $B$-bimodules:
\begin{equation}
\label{eps}
\xymatrix{
0 \ar[r] & B^+ \ar[r] & B \ar[r]^{\epsilon} & k \ar[r] & 0,
}
\end{equation}
where $\epsilon: B \to k$ is augmentation map and $B^+ \cong {}_{\psi_1}A_{\psi_2}$ as $B$-bimodule.
We have natural functors: $\psi_{1*}: {\rm A - Mod} \to {\rm B - Mod}$ and $\psi_{2*}: {\rm A - Mod} \to {\rm B - Mod}$. Also, we have functors: $\psi^!_i: {\rm B - Mod} \to {\rm A - Mod}, i = 1,2$ defined by formulas:
\begin{equation}
\psi^!_1, \psi^!_2: V \mapsto {\rm Hom}_B(A,V),
\end{equation}
where $A$ is endowed with structure of $B$-modules via $\psi_1$ and $\psi_2$ respectively.
Also, there are two functors $\psi^*_i: {\rm B - Mod} \to {\rm A - Mod}, i = 1,2$ defined by:
\begin{equation}
\psi^*_1, \psi^*_2: V \mapsto A \otimes_B V,
\end{equation}
where $A$ is a right $B$-module via $\psi_1$ and $\psi_2$ respectively.
Using \cite{BZ}, we have the natural transformation of functors: $\mu: \psi^!_1 \to \psi^*_2$. $\mu$ is defined as follows. Fix $V \in {\rm B - mod}$. Let $\rho: B \to {\rm End}_k(V)$ be a corresponding representation. In this case $\mu_V$ is defined by:
\begin{equation}
\mu_V: a \otimes v \mapsto \phi_{a \otimes v} \in {\rm Hom}_B(A,V),
\end{equation}
where $\phi_{a \otimes v}(a') = \rho(a' \cdot_A a)v$

Also, we have natural functor: $\epsilon_*: {\rm k - Mod} \to {\rm B - Mod}$ and adjoint functors: $\epsilon^!, \epsilon^*: {\rm B - Mod} \to {\rm k - Mod}$ defined by rules: $\epsilon^!: V \mapsto {\rm Hom}_B(k,V)$ and $\epsilon^*: V \mapsto k \otimes_B V$.

We have the following diagram of functors:
\begin{equation}
\label{recol}
\xymatrix{
{\rm k-mod} \ar[rrr]^{\epsilon_*}
&&& {\rm B -mod} \ar@/_2pc/[lll]_{\epsilon^*} \ar@/^2pc/[lll]^{\epsilon^!} \ar[rrr]^{\psi^!_1 = \psi^*_2}
&&& {\rm A - mod} \ar@/_2pc/[lll]_{{\psi_1}_*} \ar@/^2pc/[lll]^{{\psi_2}_*}
}
\end{equation}

Direct checking shows us that ${\rm Im}\epsilon_* = {\rm Ker}\psi^*_2$. Also, functors ${\psi_1}_*$ and ${\psi_2}_*$ are exact iff $B^+$ is right and left projective $B$-module. One can show that $\psi^!_1 {\psi_1}_* \cong {\rm Id}$ and ${\psi^*_2} {\psi_2}_* \cong {\rm Id}$. Thus, in this case functors ${\psi_1}_*$ and ${\psi_2}_*$ are fully faithful functors.
Adjointness of functors is well-known and thus, recollement situation of (\ref{recol}) is clear.

{\bf Definition.}
Consider unital associative algebra $A$. Fix element $\Delta \in A$, $B = \widehat{A}_{\Delta}$. %Denote by $B$ the augmented homotope $\widehat{A}_{\Delta}$.
We will say that $\Delta$ is {\it well-tempered} iff
\begin{enumerate}
\item{ideal of augmentation $B^+$ is right and left projective $B$-module}
\item{multiplication map $B^+ \otimes_k B^+ \to B^+$ is surjective, i.e. $A \Delta A = A$.}
\end{enumerate}

Note that if $A$ is finite-dimensional then one can deduce second condition from first condition.
\begin{Proposition}
Consider finite-dimensional algebra $A$, fix $\Delta \in A$ and $B = \widehat{A}_{\Delta}$. If $B^+$ is projective right and left $B$-module then multiplication map $B^+ \otimes_k B^{+} \to B^+$ is surjective and hence, $A \Delta A = A$.
\end{Proposition}
\begin{proof}
It is evident that algebra $B$ is finite-dimensional. Assume that $B$ has only $s$ simple $B$-modules. It is well-known that there are $s$ indecomposable projective $B$-modules $P_i, i = 1,...,s$.
Since $B^+$ is projective finite generated $B$-module, there is the following decomposition: $B^+ = \oplus^s_{i=1} P^{\oplus n_i}_i$ for some $n_i \ge 0$.
Moreover, there are idempotents $e_i, i = 1,...,s$ such that $P_i = B e_i, i = 1,...,s$.

Thus, $k \otimes_B B^+ = \oplus^s_{i=1} (k \otimes_B Be_i)^{\oplus n_i}$. One can show that $k \otimes_B Be_i = 0$ and hence, $k \otimes_B B^+ = 0$.
Tensoring sequence (\ref{eps}) by $B^+$ and using projectivity of $B^+$, we get that $B^+ \otimes_B B^{+} \cong B^+$ as $B$-bimodules.
Further, consider multiplication map $m: B \otimes_k B \to B$. Tensoring it by $B^+$ from left and right side, we get that $B^+ \otimes_k B^+ \to B^+ \otimes_B B^+ = B^+$ is surjective.
\end{proof}

In section \ref{main} we will prove that first and second conditions are equivalent for finite-dimensional algebras.

\begin{Corollary}\cite{BZ}
Let $A$ be a unital commutative algebra. Element $\Delta \in A$ is well-tempered iff $\Delta$ is invertible.
\end{Corollary}
\begin{proof}
In this case $A \Delta A = A\Delta = A$ and hence, $\Delta$ is invertible. Further, we have the following isomorphism of algebras $B = \widehat{A}_\Delta \cong k \oplus A$ and the following decomposition of unit: $1 = (1-1_A) + 1_A$. Thus, $A$ is projective $B$-module.
\end{proof}

\begin{Corollary}\cite{BZ}
Consider matrix algebra $M_n(k)$. Element $\Delta \in M_n(k)$ is well-tempered iff $\Delta \ne 0$.
\end{Corollary}

\subsection{Well-tempered elements of finite-dimensional algebras and properties of homotopes.}
\label{main}

Consider finite-dimensional associative algebra $A$. Denote by $R(A)$ the Jacobson radical of $A$. By Maltzev - Wedderburn theorem, we have the following decomposition of algebra $A = R(A) \oplus S$, where $S \cong A/R(A)$ is a semisimple algebra. Of course, $S = \oplus^t_{i=1}M_{n_i}(k)$.
Denote by $U(A)$, $U(R)$ and $GL(S)$ the group of units of $A$, the subgroup of $U(A)$ consisting of elements $1 + r, r \in R$ and the product $\times^t_{i=1}GL_{n_i}(k)$ respectively.
We have the following trivial proposition:
\begin{Proposition}
U(R) is a normal subgroup of $U(A)$. Group $U(A)$ is a semi-direct product of $GL(S)$ and $U(R)$, i.e. $U(R)$ is an unipotent radical of $U(A)$.
\end{Proposition}

Using proposition \ref{invele}, we obtain that if $\Delta_1, \Delta_2$ are in the same double coset $U(A) \backslash A / U(A)$ then $A_{\Delta_1} \cong A_{\Delta_2}$ and $\widehat{A}_{\Delta_1} \cong \widehat{A}_{\Delta_2}$. It is clear that if $\Delta_i, i = 1,2$ are in the same double coset $U(A) \backslash A / U(A)$ then $\Delta_1$ is well-tempered iff $\Delta_2$ is so.
Study suitable view of elements in a double coset $U(A) \backslash A / U(A)$.
%For any element $\Delta \in A$ we have t decomposition $\Delta = s + r, s \in S, r \in R(A)$. %There is a decomposition:
%\begin{equation}
%%\label{decdel}
%\Delta = s + r, s \in S, r \in R(A).
%\end{equation}

%The following technical lemma shows us that $s$ and $r$ from (\ref{decdel}) could be picked up in the following suitable manner:
\begin{Lemma}
\label{convform}
In any double coset $U(A)\backslash A/U(A)$ there is an element $x = s + r$, where $s$ and $r$ satisfy to relations:
\begin{itemize}
    \item{$s^2 = s$,}
    \item{$sr = rs = 0$.}
\end{itemize}
\end{Lemma}

%For any element $\Delta \in A$ there are elements $g_1,g_2 \in U(A)$ such that $g_1 \Delta g_2 = s' + r', s' \in S, r' \in R$, where $s'$ and $r'$ satisfy to relations:
%\begin{itemize}
%    \item{$s'^2 = s'$,}
%    \item{$s'r' = r's' = 0$.}
%\end{itemize}
%\end{Lemma}
\begin{proof}
Consider element $x_1 = s_1 + r_1, s \in S, r \in R$. First statement is easy. Actually, there are elements $h_1, h_2 \in GL(S)$ such that $h_1 s_1 h_2 = s$, where $s^2 = s$. $x_2 = h_1 x_1 h_2 = s + r_1, r_2 = h_1r_1h_2$.
It is evident that action of $U(R)$ on $S$ is trivial. Using action of $U(R)$, one can show that element $r_1$ can be transformed into the element $r$ satisfying to second condition of lemma.
Actually, $r_2 = sr_2 + (1-s)r_2$ and $x_2 = s(1+r_2) + (1-s)r_2$. Direct calculations show us that $x_3 = x_2(1+r_2)^{-1} = s + (1-s)r_2(1+r_2)^{-1}$.
%Further, we can change $f_1, f_2 \in U(R)$ such that
Denote by $r_3$ the element $(1-s)r_2(1+r_2)^{-1}$. It is easy that $sr_3 = 0$.
Analogously, $x_3 = (1+r_3)s + r_3(1-s)$. Consider element $x = (1+r_3)^{-1}x_3 = s + (1+r_3)^{-1}r_3(1-s)$. Since $r_3$ and $(1+r_3)^{-1}$ commute we get that
$x = s + r_3(1+r_3)^{-1}(1-s)$. Denote by $r$ the element $r_3(1+r_3)^{-1}(1-s)$. Thus, $sr = rs = 0$.

%We have the following identity for $r_1$: $r' = s'r'+(1-s')r'$. Thus, $\Delta' = s'(1+r') + (1-s')r'$. Consider $\Delta' (1+r')^{-1} = s' + (1-s')r'(1+r')^{-1}$. Denote by $\Delta''$ and $r''$ the elements $\Delta'(1+r)^{-1}$ and $(1-s')r'(1+r')^{-1}$ respectively. In this case $\Delta'' = s' + r''$ and $s' r'' = 0$. Analogously, consider decomposition $r'' = r''s'+ r'' (1-s')$. Moreover, $\Delta'' = (1+r'')s' + r'' (1-s')$ and $(1+r'')^{-1}\Delta'' = s' + (1+r'')^{-1}r''(1-s') = s' + r''(1+r'')^{-1}(1-s')$. Denote by $r'''$ the element $r''(1+r'')^{-1}(1-s')$. Thus, $\Delta''' = (1+r'')^{-1} \Delta' = s' + r'''$, where $s'^2 = s'$ and $s' r''' = r''' s' = 0$.
\end{proof}
We will say that $\Delta$ is {\it suitable} if decomposition $\Delta = s+r$ satisfy to lemma \ref{convform}. Make some useful remark on suitable $\Delta \in A$.
We have irreducible representations $\varrho_i, i = 1,...,t$ of $A$. Thus, we have the decomposition of
\begin{equation}
\label{decs}
s = \sum^t_{i=1} s_i,
\end{equation}
where $s_i = \varrho(\Delta) \in M_{n_i}(k)$ and $s_i s_j = 0$ for any $i \ne j$. If $\Delta$ is suitable then $s^2_i = s_i, i = 1,...,t$. Also, we have the decomposition of unity of $A$ of the following type: $1 = \sum^t_{i=1}e_i$, where $e_i, i = 1,...,t$ are identity elements of $M_{n_i}(k)$. Pickup the decomposition of $e_i$ into sum of orthogonal idempotents: $e_i = \sum^{n_i}_{j=1}e^j_i$ such that
\begin{equation}
\label{decsi}
s_i = \sum_{j \in I_i}e^j_i
\end{equation}
for $I_i \subseteq \{1,...,n_i\}$. It is easy that ${\rm rank}\varrho_i(s) = |I_i|$ and $s_i e^j_i = e^j_i s_i = e^j_i$ for any $i$ and $j \in I_i$.
%Denote by

\begin{Theorem}
Consider finite-dimensional algebra $A$ and element $\Delta$ such that $A \Delta A = A$ then $\Delta$ is well-tempered element of $A$.
\end{Theorem}
\begin{proof}
Consider homotope $B = \widehat{A}_{\Delta}$. We have to prove that augmentation ideal $B^+$ is right and left projective $B$-module.
Let us prove that $B^+ \cong {}_{\psi_1}A$ is left projective $B$-module.
Using lemma \ref{convform}, we can consider only suitable element $\Delta$.
Consider decomposition of $\Delta = s + r$, where (\ref{decs}) is a decomposition of $s$ and (\ref{decsi}) is a decomposition of any $s_i$ in (\ref{decs}).
Condition $A \Delta A = A$ is equivalent to $I_i \ne \emptyset$ for any $i = 1,...,t$ in (\ref{decsi}).
It is well-known the following isomorphisms of $A$-modules: $A = \oplus^t_{i=1}Ae_i$, $Ae_i = \oplus^{n_i}_{j=1}A e^j_i, i = 1,...,t$ and $Ae^j_i \cong Ae^l_i$ for any $i = 1,...,t$ and $j,l \in \{1,...,n_i\}$.
For any $i \in \{1,...,t\}$ pickup $j_i \in I_{i}$. Consider elements $e^{j_i}_{i}, i \in \{1,...,t\}, j_i \in I_i$. In this case $e^{j_i}_{i} *_{\Delta} e^{j_i}_{i} = e^{j_i}_{i}$ and, hence, $B *_{\Delta} e^{j_i}_{i}$ is projective $B$-module for any $i \in \{1,...,t\}$ and $j_i \in I_i$. Direct checking shows that $B *_{\Delta} e^{j_i}_{i} \cong Ae^{j_i}_{i}$ as $B$-modules. Thus, $Ae_i \cong (B *_{\Delta} e^{j_i}_i)^{\oplus n_i}$ and $A = \oplus^t_{i=1} (B *_{\Delta} e^{j_i}_i)^{\oplus n_i}$.
\end{proof}

If $\Delta$ is well-tempered element of $A$ then we have the following inequality for global dimension: ${\rm gl.dim}(B) \le \max(2,{\rm gl.dim}(A))$ (cf. \cite{BZ}).
\begin{Corollary}
If $\Delta$ is not well-tempered element of finite-dimensional algebra $A$ then algebra $B = \widehat{A}_{\Delta}$ has infinite global dimension.
\end{Corollary}
\begin{proof}
If $\Delta$ is not well-tempered element then multiplication map $B^+ \oplus_k B^+ \to B^+$ is not surjective. Denote by $M \subset B^+$ the image of $m$.
It is easy that $M$ is $B$-module and quotient $B^+/M$ is a direct sum of several copies of trivial $B$-modules.
Further, recall the famous result of Igusa \cite{Igu}: if $B$ is finite-dimensional algebra of finite global dimension then ${\rm Ext}^1_B(V,V) = 0$ for any simple $b$-module.
Let us prove that ${\rm Ext}^1_B(k,k) \ne 0$. Using exact sequence (\ref{eps}), we get that ${\rm Ext}^1_B(k,k) = {\rm Hom}_B(B^+,k)$. It is clear that if $\Delta$ is not well-tempered element then $B^+/M \ne 0$ and hence, $0 \ne {\rm Hom}_B(B^+/M,k) \to {\rm Hom}_B(B^+,k)$ is injective. Thus, ${\rm Hom}_B(B^+,k) \ne 0$ and we get the required statement.
\end{proof}

Further, study some properties of radical and representations of homotopes.
The following proposition describes the connection between radicals of algebra $A$ and augmented homotope $\widehat{A}_{\Delta}$:

\begin{Proposition}
\label{radhom}
Let $A$ be a finite-dimensional associative algebra. Fix $\Delta \in A$. %Denote by $B$ the homotope $\widehat{A}_{\Delta}$.
Then
\begin{itemize}
%\item{Any two-sided ideal of $A$ is a two-sided ideal of $B = \widehat{A}_{\Delta}$.}
\item{$R(A) \subseteq R(B)$, and}
\item{$R(A) = R(B)$ iff $\Delta$ is invertible element of $A$.}
\end{itemize}
\end{Proposition}
\begin{proof}
%First statement is trivial.
Direct calculations show us that $R(A)$ as two-sided $B$-ideal is nilpotent. Thus, $R(A) \subseteq R(B)$.
If $\Delta \in A$ is invertible then $B \cong A \oplus k$. Thus, if $\Delta$ is invertible then $R(A) = R(B)$.

%Recall that by Maltzev - Wedderburn theorem, we have the decomposition of $A$ into direct sum:
%$A = R(A) \oplus S$, where $S$ is a subalgebra of $A$ such that $S \cong A/R(A)$.
%Consider element $\Delta \in A$. In this case we have the following decomposition: $\Delta = s + r, s \in S, r \in R(A)$.
%Using proposition \ref{invele} and lemma \ref{convform}, we can assume that $\Delta = s + r$ is suitable.
Fix $\Delta = s+r, s \in s, r \in R(A)$.
Note that if $\Delta \in A$ is a zero divisor, then $s$ is a zero divisor of $S$.
In this case there is an element $s_1 \in S$ such that $s s_1 = s_1 s = 0$. Consider subspace $J \subset B$ generated by $R(A)$ and $s_1$.
Direct calculations show us that $B *_{\Delta} J \subseteq J$ and $J *_{\Delta} B \subseteq J$ and $J*_{\Delta}J \subseteq R(A)$.

%Fix $j = r_1 + \alpha s_1 \in J, \alpha \in k$.
%It is easy that $x *_{\Delta} j = x \Delta (r' + \alpha \cdot s_1) = x (s' + r) (r_1 + \alpha s_1) = x s'r_1 + xrr_1 + \alpha x r s_1 \in R(A)$ for any $x \in A$. It is clear that $B *_{\Delta} J \subseteq J$. Analogously, $J *_{\Delta} B \subseteq J$. Thus, $J$ is a two-sided ideal of $B$.

%Let us prove that $J$ is nilpotent.
%It easy that $s_1 *_{\Delta} s_1 = s_1 r s_1 \in R(A)$ and $s_1 *_{\Delta} x = s_1 r x \in R(A)$, $x *_{\Delta} s_1 = x r s_1 \in R(A)$ for any $x \in R(A)$. Thus $J^2 \subseteq R(A)$, and hence, $J$ is nilpotent and $J \subseteq R(B)$.
\end{proof}

Further, consider irreducible representations of homotopes. It is clear that algebra $B$ has trivial representation given by augmentation.

Consider element $\Delta \in A$. There are $t$ functions ${\rm rank}_i: A \to {\mathbb N}_0, i = 1,...,t$ defined by rule: ${\rm rank}_i(a): = {\rm rank}\varrho_i(a), i = 1,...,t$, where $\varrho_i, i = 1,...,t$ are irreducible representations of $A$.
It is clear that if $\Delta_1$ and $\Delta_2$ are in the same double coset $U(A) \backslash A / U(A)$, then ${\rm rank}_i(\Delta_1) = {\rm rank}_i(\Delta_2), i = 1,...,t$.
%For simplicity, denote by $r_i, i = 1,...,t$ the values ${\rm rank}_i(\Delta)$.
Denote by $I(\Delta)$ the subset of $\{1,...,t\}$ consisting of indexes $i$ such that ${\rm rank}_i(\Delta) > 0$.

\begin{Corollary}
Let $A$ be a finite-dimensional algebra. Fix $\Delta \in A$. Assume that ${\rm rank}_i(\Delta) = r_i, i = 1,...,t$. Consider augmented homotope $B = \widehat{A}_{\Delta}$. There is a bijection between the set of irreducible representations of $B$ and $I(\Delta) \cup \{\epsilon\}$, where $\epsilon$ is a trivial representation. Dimensions of irreducible representations are $r_i$ for $i \in I(\Delta)$ and $1$ for $\epsilon$. %and ${\rm dim}_k \epsilon = 1$ for trivial.
\end{Corollary}
\begin{proof}
Consider decompositions of $A = S(A) \oplus R(A)$ and $B = S(B) \oplus R(B)$, where $S(A)$ and $S(B)$ are semisimple parts of $A$ and $B$ respectively.
Also, we have decomposition of $\Delta = s + r, s \in S(A), r \in R(A)$.
As we know, $R(A) \subseteq R(B)$. Using proposition \ref{commhom}, we know that operation of taking quotient by two-sided ideal and operation of taking homotope are commuting. Thus, we get that $S(B) = \widehat{S(A)}_{s}$, where $s$ is semisimple part of $\Delta$. Further, we can consider only matrix algebra. The rest is a direct checking.

%Of course, we can assume that $\Delta$ is suitable.
%Thus, $\Delta = s + r$ and $s = \sum_{i \in I(\Delta)} s_i$, where $\pi_i(\Delta) = s_i$ for $i \in Ess(\Delta)$. Also, $s^2_i = s_i, i \in I(\Delta)$, $s_i s_j %= 0$ for $i \ne j$ and $s_i r = r s_i = 0$. Direct calculations show us that $s_i *_{\Delta} s_i = s_i$ and $s_i *_{\Delta} s_j = 0$ for $i \ne j$. Thus, we get %the decomposition of unit $1 \in B$ into sum of orthogonal idempotents:
%$$
%1 = \sum_{i \in I(\Delta)}s_i + (1 - \sum_{i \in I(\Delta)}s_i).
%$$

%It is easy that $s_i S s_i = M_{r_i}(k) \subset S$.
%One can deduce tat $s_i *_{\Delta} B *_{\Delta} s_i \cong M_{r_i}(k)$. Standard arguments from theory of artinian algebras give us the result.
\end{proof}

%Recall the construction of quiver $QA = (Q_0,Q_1)$ of finite-dimensional algebra $A$.
%Let $R(A)$ be a Jacobson's radical and $S = A/R(A) = \oplus^t_{i=1}S_i$, where $S_i \cong M_{n_i}(k)$.
%Vertices of $Q$, i.e. elements of $Q_0$, correspond to irreducible $A$-modules.
%For $i, j \in Q_0$,
%let the number $t_{ij}$ of arrows from $i$ to $j$ in $QA$ be the rank of the finitely generated
%$S_j-S_i$-bimodule $S_i \cdot R(A)/R^2(A) \cdot A_i$.
%Recall that for two artinian algebras $A$ and $B$, the rank of a finitely generated $A-B$-bimodule $M$ is defined as the least cardinal number of the sets of %generators.
.

\subsection{Remark on homotopes of associative commutative algebras.}

%In this section all algebras are presumed unital associative commutative.
%In this case there is a deeply developed theory.

%\subsection{Fibre product of algebras and pushout of schemes.}

In this subsection all algebras are presumed unital commutative associative algebras.

There is a deeply developed theory in the case of commutative algebras.
Since any ideal of commutative algebras is two-sided then we can consider augmented homotope as a fibre product of algebras. Of course, in the case of noncommutative algebras homotope is not fibre product.

Firstly, recall the notion of fibre product of algebras. Consider three algebras: $A$, $B$, $R$ and morphisms $f:A \to R$ and $g: B \to R$.
{\it Fibre product} $A \times_R B$ of $A$ and $B$ over $R$ is a subalgebra of $A \times B$ consisting of pairs $(a,b)$ such that $f(a) = g(b)$.

Assume that $A$ is an integral domain. Fix element $\Delta \in A$. In this case natural morphism $\psi = \psi_1 = \psi_2:\widehat{A}_{\Delta} \to A$ is an immersion and $\psi(\widehat{A}_{\Delta}) = k \cdot 1 + A\Delta$. It is easy that $\widehat{A}_{\Delta}$ has a structure of a fibre product of some algebras. %Actually, consider algebra $A$.
%Fix element $\Delta \in A$. Put $B = k$, $R = A/A\Delta$, where $A\Delta$ is an ideal of $A$ generated by $\Delta$
$\widehat{A}_{\Delta} = A \times_R k$, where $R = A/A\Delta$ and $k \subseteq R$ is a scalar subalgebra.

Recall some results on fibre product of algebras:

\begin{Theorem}(Ogoma \cite{Og})
Consider fibre product $B' = A' \times_A B$ with morphisms: $f: A' \to A$ and $g: B \to A$. Let $I$ and $J$ be a kernels of $f$ and $g$ respectively.
Denote by $C$ the subalgebra $f(A') \cap g(B) \subseteq A$. Assume that $A'$, $B$ and $A$ are noetherian algebra.
Algebra $B' = A' \times_A B$ is noetherian if and only if
\begin{itemize}
    \item{$C$ is a noetherian algebra}
    \item{$I/I^2$ and $J/J^2$ are finite-generated $C$-modules.}
\end{itemize}
\end{Theorem}

\begin{Proposition}(\cite{FC},\cite{Sch})
Assume that morphism $f: A' \to A$ is surjective. In this case ${\rm Spec} B' = {\rm Spec}A' \sqcup_{{\rm Spec}A} {\rm Spec}B$, i.e. a pushout of schemes ${\rm Spec} B$ and ${\rm Spec} A'$ over ${\rm Spec} A$.
Also, ${\rm Spec}A' \sqcup_{{\rm Spec}A} {\rm Spec}B \setminus {\rm Spec}B \cong {\rm Spec}A' \setminus {\rm Spec}A$.
\end{Proposition}

We have the following commutative diagram:
\begin{equation}
\xymatrix{
0 \ar[r] & I \ar[r]\ar[d]^{=} & B' \ar[r]\ar[d] & B \ar[r]\ar[d]^{g} & 0\\
0 \ar[r] & I \ar[r] & A' \ar[r]^{f} & A \ar[r] & 0
}
\end{equation}

Geometrically, ${\rm Spec}B'$ is a "cutting" of subscheme ${\rm Spec}A \subset {\rm Spec}A'$ and paste ${\rm Spec} B$ instead of ${\rm Spec}A$. It is easy that subscheme ${\rm Spec}B \subset {\rm Spec}B'$ is given by ideal $I$. Thus, conormal sheaf to ${\rm Spec}A \subset {\rm Spec}A'$ and conormal sheaf to ${\rm Spec}B \subset {\rm Spec}B'$ are the same and isomorphic to $I/I^2$. It is easy that if $I/I^2$ is not a finite-generated $B$ - module then ${\rm Spec}B'$ is not noetherian. Ogoma proved that converse statement is true.
Formulate the following evident statements:
\begin{Corollary}
\label{surj}
Consider noetherian algebra $A$ and element $\Delta$. In this case natural morphism: ${\rm Spec}A \to {\rm Spec}\widehat{A}_{\Delta}$ is surjective birational.
\end{Corollary}

\begin{Corollary}
Assume that $A$ is a noetherian algebra. Fix element $\Delta \in A$. Consider homotope $\widehat{A}_{\Delta}$ and denote by $x$ the point of ${\rm Spec}\widehat{A}_{\Delta}$ corresponding to augmentation $\epsilon$. Then we have the following properties of ${\rm Spec}\widehat{A}_{\Delta}$:
\begin{itemize}
\item{Algebra $\widehat{A}_{\Delta}$ is noetherian if and only if scheme ${\rm Spec}(A/A\Delta)$ has dimension zero. In particular, if $\widehat{A}_{\Delta}$ is noetherian then ${\rm Spec}A$ is an affine scheme of dimension at most $1$.}
\item{Consider noetherian algebra $A$ and element $\Delta \in A$.
If ${\rm dim}{\rm Spec}A > 1$ then tangent space of ${\rm Spec}\widehat{A}_{\Delta}$ at point $x$ is infinite. If ${\rm dim}{\rm Spec}A = 1$ and ${\rm dim}(A/A\Delta) > 1$ then ${\rm Spec}\widehat{A}_{\Delta}$ is singular curve and $x$ is a singular point of ${\rm Spec}\widehat{A}_{\Delta}$}
\end{itemize}
\end{Corollary}

\begin{Example}(Ogoma \cite{Og}, Schwede \cite{Sch}, Beil \cite{Beil} and many others)
\label{exampl}
Consider algebra $A = k[x,y]$ and element $\Delta = x$. Using theorem of Ogoma, we get that $B = \widehat{A}_{\Delta}$ is non-noetherian and isomorphic to $k[x,xy,xy^2,...] \subset k[x,y]$.
%Note that ${\rm Spec}\widehat{A}_{\Delta}$ is a constructible subset of $k^2$.
%There is the following description of ${\rm Spec}B$.
%Denote by $B_{n}$ the subalgebra $k[x,xy,...,xy^n] \subset k[x,y]$. It is easy that $B_n \subset B$. It is easy that $B_1 \cong k[x,y]$, $B_i \subset B_{i+1}$ and %$B and hence, there is a natural morphism: $i:{\rm Spec}B \to k^2 = {\rm Spec}B_1$ which is an open embedding. We have the following commutative diagram:
%\begin{equation}
%\xymatrix{
%{\rm Spec}A \ar[rr]^f\ar[rd] && {\rm Spec}B_1\\
%& {\rm Spec}B \ar[ru]^i
%}
%\end{equation}
%morphism $f:k^2 = {\rm Spec}A \to k^2 = {\rm Spec}B_1$ is given by formula: $(x,y) \mapsto (x, xy)$.
%One can show that $f(k^2) = {\rm Spec}B$ is a constructible subset of $k^2$ and has the following decomposition: ${\rm Spec}B = k^2 \setminus \{x = 0\} \cup %(0,0)$. Tangent space of ${\rm Spec}B$ at $(0,0)$ is infinite-dimensional.
\end{Example}

This example is a counter-example to famous Richardson's lemma in the following sense.
Recall that Richardson's lemma (see Kraft's book, \cite{Rich}) is the following statement: consider two affine varieties $X$ and $Y$. Assume that $Y$ is a normal variety. If $f: X \to Y$ is surjective birational morphism then $f$ is an isomorphism. By definition, $Y = {\rm Spec}B$ is a normal affine variety iff $B$ is noetherian integral closed ring. It is natural to ask the following question: could we dispense the condition of noethering of $B$?
The answer is no.
Actually, let $A$, $\Delta$ and $B$ be as in example \ref{exampl}. Using corollary \ref{surj}, natural morphism: ${\rm Spec}k[x,y] \to {\rm Spec}B$ is surjective birational.
We have to prove that algebra $B$ is integral closed. It is easy that $B = \varinjlim B_n$, where $B_n = k[x,xy,...,xy^n]$.
It is easy that ${\rm Spec}B_n$ is an affine cone over rational normal curve and hence, $B_n$ is integral closed. It can be shown in usual way that direct limit of integral closed rings is integral closed.
%Note the following properties: ${\rm Spec}B$ is an image of morphism: $(x,y) \mapsto (x,xy)$.
%One can show that $f(k^2) = {\rm Spec}B$ is a constructible subset of $k^2$ and has the following decomposition: ${\rm Spec}B = k^2 \setminus \{x = 0\} \cup (0,0)$.

In the end of subsection we recall the simple example on affine curves.
\begin{Example}
Consider algebra $A = k[x]$ and element $\Delta = x^2 + ax + b, a,b \in k$. One can deduce that ${\rm Spec}\widehat{A}_{\Delta}$ is a rational curve with node if roots of the polynomial $\Delta$ are simple and rational curve with cusp if $\Delta$ has double root.
\end{Example}
%Of course, these examples can be considered as affine blowdown

%Consider affine curve $C = {\rm Spec}A$, element $\Delta$ and its immersion $i: C \subset {\mathbb A}^s$. Let $\overline{C}$ be a projectivization of $C$ given by $i$, i.e. $i: \overline{C} \subset {\mathbb P}^s$ be a natural continuation of $i$. Denote by $D$ the divisor of $\overline{C}$ given by $\Delta$.
%Assume that $D_{\infty}$ the intersection of infinity devisor of ${\mathbb P}^s$ with $\overline{C}$, i.e. $\overline{C} = C \cup D_{\infty}$ and $D \cap D_{\infty} = \emptyset$. Of course, immersion $i$ curve $\overline{C}$ in ${\mathbb P}^s$ is given by complete linear system ${\rm H}^0(\overline{C}, {\cal O}(D_{\infty}))$. Consider linear system $L \subset {\rm H}^0(\overline{C}, {\cal O}(D + D_{\infty}))$ consisting of divisors $D'$ such that $D' - D \ge 0$.
%Adding to $L$ the divisor $D_{\infty}$, we get linear system $L'$. Consider morphism $f$ given by linear system $L'$. Using standard arguments, we get that $C' = f(\overline{C})$ is a compactification of ${\rm Spec}\widehat{A}_{\Delta}$ and ${\rm Spec}\widehat{A}_{\Delta} = C' \setminus f(D_{\infty})$.

 %One can check that $B = \varinjlim B_n$.  Thus, $B$ is non-noetherian integral closed ring.

\subsection{Categorial approach to homotopes of commutative algebras.}

What can we say in the case when $\Delta$ is not well-tempered? We give the partial answer on this question only in the case of commutative algebras.

Recall the notion fibre product of categories.
Let ${\cal A}_1$, ${\cal A}_2$ and ${\cal A}_3$ are categories. Let ${\cal F}: {\cal A}_1 \to {\cal A}_3$ and ${\cal G}: {\cal A}_2 \to {\cal A}_3$ are the functors. Define the category ${\cal A}_1 \times_{{\cal A}_3} {\cal A}_2$ as follows. Object of ${\cal A}_1 \times_{{\cal A}_3} {\cal A}_2$ is a triple $(M,N,\alpha)$, where $M \in Ob({\cal A}_1), N \in Ob({\cal A}_2)$ and $\alpha: {\cal F}(M) \cong {\cal G}(N)$ is an isomorphism of objects. Morphism $(M,N,\alpha) \to (M',N',\alpha')$ is a pair $(a,b)$, $a: M \to M'$ and $b: N \to N'$ are morphisms in categories ${\cal A}, {\cal B}$ respectively such that the diagram:
\begin{equation}
\xymatrix{
{\cal F}(M) \ar[rr]^{\alpha}\ar[d]^{{\cal F}(a)} && {\cal G}(N)\ar[d]^{{\cal G}(b)}\\
{\cal F}(M') \ar[rr]^{\alpha'} && {\cal G}(N')
}
\end{equation}
is commutative.
%Denote by ${\rm A - mod}$ the category of finitely generated left $A$-modules.
Consider three algebras $A'$, $A$, $B'$ and morphisms $f: A' \to A$ and $g: B' \to A$, where $f$ is surjective.
Let $B = A \times_{A'} B'$.
In this case there is a functor: $\Psi: {\rm B-Mod} \to {\cal C} = {\rm A - Mod} \times_{\rm A' - Mod} {\rm B' - Mod}$ defined by formula:
\begin{equation}
\Psi: L \mapsto (B' \otimes_{B} L, A \otimes_{B} L, can),
\end{equation}
where $can$ is a natural isomorphism: $A' \otimes_{A} A \otimes_{B} L \cong A' \otimes_{B'} B' \otimes_{B} L$. Conversely, let $N$ and $M'$ be $B'$ - module and $A'$ - module respectively. Let $\phi$ be an isomorphism $A' \otimes_{B'} N \cong A' \otimes_{A} M'$ of $A$ - modules. Denote by $M$ the $A$-module $A' \otimes_{A} M'$.

\begin{Theorem}\cite{St}
\label{thf}
Functor $\Psi': (N,M',\phi) \mapsto N \times_{\phi, M}M'$ is right adjoint to $\Psi$. Moreover, $\Psi \circ \Psi'$ is identity on ${\rm A - mod} \times_{\rm A' - mod} {\rm B' - mod}$.
Morphism $L \to \Psi' \circ \Psi(L) = B' \otimes_{B}L \times_{A' \otimes_{B}L} A \otimes_{B}L$ is surjective.
\end{Theorem}

Consider algebra $A$, ideal $I \subseteq A$ and $B = A \times_{A/I} B'$.
It is clear that we have the following exact sequence of $B$ - modules:
\begin{equation}
\label{seq}
\xymatrix{ 0 \ar[r] & B \ar[r] & A \oplus B' \ar[r] & A/I \ar[r] & 0
}
\end{equation}
Tensoring (\ref{seq}) by $B$-module $V$, we get that $\Psi' \circ \Psi(V) = {\rm Ker} (A \otimes_B V \oplus B' \otimes_B V \to A/I \otimes_B V)$.
Thus, natural morphism $V \to \Psi' \circ \Psi (V)$ is isomorphism iff ${\rm Tor}^{B}_1(A\oplus B',V) \to {\rm Tor}^{B}_1(A/I,V)$ is isomorphism.

If $B' = k$ then $A/I$ is a trivial $B$-module. Denote by $\overline{A/I}$ the complement of $k$ in $A/I$.
In this case we have the following exact sequence of $B$-modules:
\begin{equation}
\xymatrix{
0 \ar[r] & B \ar[r] & A \ar[r] & \overline{A/I} \ar[r] & 0.
}
\end{equation}
Also, if $I = (\Delta)$ is a principal ideal of integral domain $A$ generated by $\Delta$ (i.e. $B = \widehat{A}_{\Delta}$) then we have the following exact sequence:
\begin{equation}
\xymatrix{
0 \ar[r] & A \ar[r]^{j} & B \ar[r] & k \ar[r]^{\epsilon} & 0,
}
\end{equation}
where $j$ is a morphism given by rule: $a \mapsto a \Delta$. We have isomorphism of $B$-modules: ${\rm Tor}^B_{i+1}(k,V) \cong {\rm Tor}^B_i(A,V)$ for any $B$-module $V$.  Thus, we get that $W \to \Psi' \circ \Psi (V)$ is an isomorphism iff ${\rm Tor}^B_2(k,V) \to {\rm Tor}^B_1(\overline{A/I},V)$ is an isomorphism.
%In particular, if $V$ is a quotient of $B$ by principal ideal then ${\rm Tor}^B_2(W,V) = {\rm Tor}^B_2(V,W) = 0$ for any $B$ - module $W$.
%And hence, if ${\rm Tor}^B_1(\overline{A/I},V) \ne 0$ then $V \to \Psi' \circ \Psi(V)$ is not isomorphism.
Assume that $V = B/(u)$ where $u \in B$ is a non-zero element, then ${\rm Tor}^B_2(M,V) = 0$ and ${\rm Tor}^B_1(M,V) = M/(u)M$ for any $B$-module $M$.
Thus, if $M$ is a trivial $B$ - module and $u \in I$ then ${\rm Tor}^B_1(\overline{A/I},B/(u)) = \overline{A/I} \ne 0$.
In the case of homotope we have the following statement:
\begin{Proposition}
Consider algebra $A$, element $\Delta \in A$ and $B = \widehat{A}_{\Delta}$. Then $\Psi' \circ \Psi(V) \cong V$ iff ${\rm Tor}^B_1(\overline{A/(\Delta)},V) \cong {\rm Tor}^B_2(k,V)$. If $V = B/(u)$ where $u \in I$ then kernel of natural morphism: $V \to \Psi' \circ \Psi(V)$ is $\overline{A/(\Delta)}$.
\end{Proposition}

Consider category ${\cal C}$ ($ = {\rm A-Mod} \times_{{\rm A/I-Mod}} {\rm k - Mod}$) and natural functor ${\cal C} \to {\rm A - Mod}$.
One can show that ${\rm Hom}_{\cal C}(V_1,V_2) \cong {\rm Hom}_A(V_1,V_2)$ for any $V_1,V_2 \in {\cal C}$ and identify ${\cal C}$ with a full subcategory of ${\rm A - Mod}$ consisting of $A$-modules $W$ such that $A/I \otimes_A W$ is a free $A/I$-module. Roughly speaking, category ${\cal C}$ is a "cutting" of category of ${\rm A/I-Mod}$ from ${\rm A - Mod}$ and "pasting" the category ${\rm k - Mod}$ instead of ${\rm A/I - Mod}$. Of course, if $\Delta$ is well-tempered element (invertible) of $A$ then category ${\cal C}$ is a "gluing" of ${\rm A - Mod}$ and ${\rm k - Mod}$ and ${\rm B - Mod}$ is equivalent to ${\cal C}$.
%One can show that natural functor ${\cal C} \to {\rm A - Mod}$ identify is
%Also, ${\rm Hom}_{\cal C}(V_1,V_2) \cong {\rm Hom}_A(V_1,V_2)$ for any $V_1,V_2 \in {\cal C}$. Thus, ${\cal C}$ is a full subcategory of ${\rm A - Mod}$.
Using theorem \ref{thf}, we get that ${\cal C}$ is a full subcategory of ${\rm B - Mod}$.
Thus, we get the following
\begin{Proposition}
Consider commutative algebra $A$, element $\Delta \in A$ and homotope $B = \widehat{A}_{\Delta}$. In this case category ${\cal C}$ is a full subcategory of ${\rm A - Mod}$ and ${\rm B - Mod}$.
\end{Proposition}

{\bf Remark.}
Consider algebra $A = k[x]$, element $\Delta$ of degree 2 without double roots and $B = \widehat{A}_{\Delta}$. As we know, ${\rm Spec}B$ is a rational curve with node. In this case construction of ${\cal C}$ is a local version of beautiful construction of Burban and Drozd (see \cite{BD1}, \cite{BD2}). Using fibre product of categories, they classify indecomposable object in the category of coherent sheaves on rational curve with node.

.

\end{document}